\documentclass[12pt,a4paper]{article}
\usepackage{amsmath, amsthm, amssymb, verbatim, textcomp, graphicx}
\usepackage[utf8]{inputenc}
\usepackage[a4paper,left=2cm,right=2.5cm,top=3.5cm,bottom=3cm]{geometry}
\usepackage{hyperref}
\usepackage{color}
\usepackage{authblk}

\newtheorem{thm}{Theorem}[section]
\newtheorem{definition}{Definition}[section]

\def\ts{\thinspace}

\def\C{\mathbb{C}}
\def\R{\mathbb{R}}
\def\N{\mathbb{N}}

\def\u{{\bf u}}
\def\v{{\bf v}}
\def\w{{\bf w}}
\def\A{{\bf A}}

\def\E{{\bf E}}
\def\0{{\bf 0}}
\def\B{{\bf B}}

\def\I{{\bf I}}

\def\c{{\bf c}}

\def\H{{\bf H}}

\title{Spijker's example and its extension}
\author{Miklós E. Mincsovics}
\affil{{\normalsize MTA-ELTE Numerical Analysis and Large Networks Research Group, Hungary}\\
{\normalsize Budapest University of Technology and Economics, Department of Differential Equations, Hungary}}
\begin{document}
\maketitle

\begin{abstract}
Strongly and weakly stable linear multistep methods can behave very differently. The latter class can produce spurious oscillations in some of the cases for which the former class works flawlessly. The main question is if we can find a well defined property which clearly tells the difference between them. There are many explanations from different viewpoints. We cite Spijker's example which shows that the explicit two step midpoint method is unstable with respect to the Spijker norm. We show that this result can be extended for the general weakly stable case.
\end{abstract}



\section{Introduction}

This paper focuses on the stability and instability of linear multistep methods. When we introduce linear multistep methods it is unavoidable to talk about the root-condition and usually about the two types of it, which divide these methods into two classes: the weakly and strongly stable linear multistep methods. This can be found in almost every textbook about the numerical solution of ordinary differential equations, see e.g. \cite[Section 5.2.3]{AP}. The root-condition is closely related to the stability of linear multistep methods. As it is well-known stability together with consistency implies the convergence of the method. This result can be obtained in different setups, we follow the book \cite{S73} where stability, consistency and convergence are defined in a general sense forming the base of a beautiful theoretical framework. This book also gives a detailed description how to use this framework for nonlinear ODEs. Our intention is to clarify the relation of the strongly/weakly stable linear multistep methods and stability of linear multistep methods in the above mentioned setting. Spijker's example \cite{S} gave the first (negative) result about this relation. The example shows that the explicit two step midpoint method is unstable with respect to the general notion of stability if we use an unusual norm. We extend this example to the whole weakly stable class.

We organized the paper as follows.
We introduce linear multistep methods and their basic notions which are important for us, including the definition of weakly/strongly stable linear multistep methods. Then we reformulate linear multistep methods and define stability in the general sense. After this preparation we recall Spijker's example and finally we present the new result which generalizes Spijker's example. We conclude the paper with a critical remark.

\section{Stability of linear multistep methods}

Without loss of generality we consider the scalar autonomous \emph{initial value problem} (IVP)
\begin{equation}\label{IVP}
\begin{cases}
u(0) = u^{0}\ts,\\
\dot u(t) = f(u(t))\ts,
\end{cases}
\end{equation}
where $t\in[0,T],\ u^{0}\in\R$ is the initial value, $u:[0,T]\to\R$ is the unknown function and we assume
that $f$ is Lipschitz continuous.

In practice we have to use a numerical method to approximate the solution of \eqref{IVP} since finding the solution analytically is impossible in most of the cases. There are many possible choices, one is the application of a linear multistep method (LMM).
\smallskip

\emph{Linear multistep method}s can be given in the following way:
\begin{equation}\label{LMM}
\begin{cases}
u_i = c^{i}\ts,&\quad i=0,\ldots, k-1\\[8pt]
\dfrac{1}{h} \sum\limits_{j=0}^{k} \alpha_{j} u_{i-j}= \sum\limits_{j=0}^{k} \beta_{j} f(u_{i-j})\ts,&\quad i=k,\ldots, n+k-1=N\ts,
\end{cases}
\end{equation}
where $h=T/N$ is the step size, $\alpha_{j}$, $\beta_{j}\in\R$, $\alpha_{0}\neq 0$ are the coefficients of the method and the constants $c^{i}$ are some approximation of the solution on the first $k$ time levels. When these latter are known (here we do not go into the details how to determine these values since this is irrelevant to the results of the paper) the method can "run", we can calculate the next approximation and so on. To get $u_i$ which approximates the solution at the $i$-th time level $u(i\cdot h)$ we only need to know the previous $k$ approximations. Thus the formula represents a $k$-step method. Note that while  $k$ is fixed for the method $n$, $N=k+n-1$ and $h$ can vary as the grid gets finer. For shorthand notation later we will use $f_{i-j}$ for $f(u_{i-j})$.
As an example consider the explicit two step midpoint method (sometimes called leapfrog scheme in the context of parabolic PDEs)
\begin{equation}\label{midpoint}
\begin{cases}
u_i = c^{i}\ts,&\quad i=0,1\\[8pt]
\dfrac{1}{h} \left( \frac{1}{2}u_i-\frac{1}{2}u_{i-2}\right) = f_{i-1}\ts,&\quad i=2,\ldots, n+k-1=N
\end{cases}
\end{equation}
which plays the main role in Spijker's example.

The \emph{first characteristic polynomial} associated to \eqref{LMM} is defined as
\begin{equation*} %\label{ro}
\varrho(z)=\sum\limits_{j=0}^{k} \alpha_{j} z^{k-j}\ts.
\end{equation*}
Usually, two types of root-conditions are defined. These are presented below.
\begin{definition}\label{d:rootcondition}
The method is said to be \emph{strongly stable} if for every root $\xi_i\in\C$ of the first characteristic polynomial $|\xi_i|< 1$ holds except $\xi_{1}= 1$, which is a simple root.

A not strongly stable method is said to be \emph{weakly stable} if for every root $\xi_i\in\C$ of the first characteristic polynomial $|\xi_i|\leq 1$ holds and if $|\xi_i|= 1$ then it is a simple root, moreover $\xi_{1}= 1$.
\end{definition}
We note that sometimes these are defined slightly differently. The two main possible differences are the following.
First, not requiring that $\xi_{1}= 1$ holds. Second, the weakly stable class containing the strongly stable class. Our reason not to vote for this option is that we want to distinguish clearly between the two.

Roughly speaking being weakly (or strongly) stable means that applying a method for $\dot u(t)=0$ the approximation remains bounded which is an understandable requirement.

The explicit two step midpoint method is weakly stable since its first characteristic polynomial is $\frac{1}{2}\left( z^{2}-1 \right) $ with roots $z=\pm 1$.

In the weakly stable case we have another root at the boundary of the unit circle which could cause problems in some of the cases. One type of explanation about the difference between weakly and strongly stable LMMs tries to exploit this fact directly, see eg. \cite[Example 5.7]{AP}. Our approach is different.
\smallskip

In the following we rewrite LMMs \eqref{LMM} into the form for which we can define stability in the general sense. A method can be represented with a sequence of operators $F_N: \mathcal{X}_N\to \mathcal{Y}_N$, where $\mathcal{X}_N$, $\mathcal{Y}_N$ are $k+n$ dimensional normed spaces with norms $\left\|\cdot\right\|_{\mathcal{X}_N}$, $\left\|\cdot\right\|_{\mathcal{Y}_N}$ respectively and
\begin{equation*} %\label{mLMM}
(F_N(\u_N))_i=
\begin{cases}
u_i - c^{i}\ts,& \quad i=0,\ldots, k-1\\[8pt]
\dfrac{1}{h} \sum\limits_{j=0}^{k} \alpha_{j} u_{i-j} - \sum\limits_{j=0}^{k} \beta_{j} f(u_{i-j})\ts,& \quad i=k,\ldots, n+k-1=N\ts.
\end{cases}
\end{equation*}
Finding the approximating solution means that we have to solve the non-linear system of equations $F_N(\u_N) =\0$. $F_N$ can be represented in the following way:
\begin{equation*}
F_N(\u_N) =\A_N \u_N -\B_N f(\u_N)-\c_N\ts,
\end{equation*}
where $\u_{N}=(\u_{k},\u_{n})^{T}=(u_0,\ldots,u_{k-1},u_k,\ldots,u_{n+k-1})^{T}\in \R^{k+n}$, $\u_{k}\in \R^{k}$, $\u_{n}\in \R^{n}$, \\
$f(\u_N)= (f(u_0),f(u_1),\ldots,f(u_{n+k-1}))^{T}\in \R^{k+n}$, $\c_n= (c^{0},c^{1},\ldots,c^{k-1},0,\ldots,0)^{T}\in \R^{k+n}$,
$$
\A_N =\left( \begin{array}{cc}
        \I & \0  \\
        \A_{k} & \A_{n} \\
        \end{array} \right) \ts,
\qquad
\B_N =\left( \begin{array}{cc}
        \0 & \0  \\
        \B_{k} & \B_{n} \\
        \end{array} \right) \ts,
$$
where $\I\in \R^{k\times k}$ is the identity matrix, $\A_{k},\B_{k}\in \R^{n\times k}$, $\A_{n},\B_{n}\in \R^{n\times n}$,
$$
\A_{k} =\dfrac{1}{h}\left( \begin{array}{cccc}
        \alpha_k &  \ldots & \alpha_2 & \alpha_1 \\
        0 & \alpha_k &  \ldots  & \alpha_2 \\
        \vdots &  \ddots & \ddots & \vdots \\
        0 &  \ldots & \ldots & \alpha_k \\
        0 &  \ldots & \ldots & 0 \\
        \vdots  & \ddots & \ddots & \vdots \\
        0 &  \ldots & \ldots & 0 \\
        \end{array} \right)
\quad
\A_{n} =\dfrac{1}{h}\left( \begin{array}{cccccc}
        \alpha_0 & 0 & \ldots & \ldots & \ldots & 0 \\
        \alpha_1 & \alpha_0 & 0 & \ldots & \ldots & 0 \\
        \alpha_2 & \alpha_1 & \alpha_0 & 0 & \ldots & 0 \\
        \vdots & \ddots & \ddots & \ddots & \ddots & \vdots \\
        \vdots & \ddots & \ddots & \ddots & \ddots & \vdots \\
        0 & \ldots & 0 & \alpha_k & \ldots & \alpha_0
        \end{array} \right)
$$
and $\B_{k}$, $\B_{n}$ are the same as $\A_{k}$, $\A_{n}$, except that we have to omit the $\frac{1}{h}$ factor and the $\alpha$-s have to be changed to $\beta$-s.

\begin{definition}\label{d:stab}
We call a method \emph{stable in the norm pair} $\left( \left\|\cdot\right\|_{\mathcal{X}_n} \ts, \left\|\cdot\right\|_{\mathcal{Y}_n}\right)$ if for all IVP \eqref{IVP} $\exists S\in\R$ and $\exists N_0\in\N$ such that $\forall N\geq N_0 $\ts, $\forall \u_N, \v_N\in \R^{k+n}$ the estimate
\begin{equation}\label{Sstability}
\left\| \u_N- \v_N \right\|_{\mathcal{X}_N} \leq S \left\| F_N (\u_N)- F_N (\v_N) \right\|_{\mathcal{Y}_N}
\end{equation}
holds.
\end{definition}
To define stability in this way has a definite profit. It is general in the sense that it works for almost every type of numerical method approximating the solution of ODEs and PDEs as well. Convergence can be proved by the popular recipe "consistency + stability = convergence"
$$\left\| \varphi_N (\bar{u}) - \bar{\u}_N \right\|_{\mathcal{X}_N}\leq S \left\| F_N (\varphi_N(\bar{u}))- F_N (\bar{\u}_N) \right\|_{\mathcal{Y}_N} =
S \left\| F_N (\varphi_N(\bar{u})) \right\|_{\mathcal{Y}_N}\to 0\ts,
$$
where $\bar{u}$, $\bar{\u}_N$ denote the solution of the original problem \eqref{IVP} and the approximating problem $F_N (\u_N)=\0$ respectively, $\varphi_N:\mathcal{X}\to \mathcal{X}_N$ are projections from the normed space where the original problem is set, thus $\varphi_N (\bar{u}) - \bar{\u}_N$ represents the error (measured in $\mathcal{X}_N$). Finally, $\left\| F_N (\varphi_N(\bar{u})) \right\|_{\mathcal{Y}_N}\to 0$ is exactly the definition of consistency in this framework. We note that the existence of $\bar{\u}_N$ (from some index) is also the consequence of stability, see \cite[Lemma 24. and 25.]{FMF}, cf. \cite[Lemma 1.2.1]{S73}. There are many versions of Definition \ref{d:stab} which are requiring the stability estimate only in some neighbourhood, see \cite{FMF}, but as we defined it is satisfactory for the IVP \eqref{IVP}.

In the following we introduce norm pairs which are interesting for us. We start with some norm notations: for $k\in \N$ fixed, $\u_{N}\in \R^{k+n}$ the $k\infty$ norm is defined as
$$\left\|\u_{N}\right\|_{k\infty}=\max_{0\leq i\leq k-1}|u_{i}|+\max_{k\leq i\leq N}|u_{i}|\ts,$$
thus $\left\|\u_{N}\right\|_{k\infty}=\left\|\u_{k}\right\|_{\infty}+\left\|\u_{n}\right\|_{\infty} \ts.$
While the $k$--Spijker-norm is defined as
$$\left\|\u_{N}\right\|_{k\$}=\max_{0\leq i\leq k-1}|u_{i}|+h\max_{k\leq l\leq N}\left| \sum\limits_{i=k}^{l} u_{i}\right| \ts.$$
Using the notation $\left\|\u_{n}\right\|_{\$}=h\max_{k\leq l\leq N}\left| \sum\limits_{i=k}^{l} u_{i}\right|$ the $k$--Spijker-norm can be expressed as
$ \left\|\u_{N}\right\|_{k\$}=\left\|\u_{k}\right\|_{\infty}+\left\|\u_{n}\right\|_{\$}\ts.$ Introducing another notation, the Spijker-norm can be given in a useful way which will be presented in the following.

First, we introduce
$\E_{n}\in \R^{n\times n}$
$$
\E_{n}=
\dfrac{1}{h}\left( \begin{array}{ccccc}
                  1 & 0 & \ldots & \ldots & 0 \\
                  -1 & 1 & 0 & \ldots & 0 \\
                  0 & -1 & 1 & 0 & \vdots \\
                  \vdots & \ddots & \ddots & \ddots & \vdots \\
                  0 & \ldots & 0 & -1 & 1
                  \end{array} \right) \quad  \mbox{for which} \quad
                  \E_{n}^{-1}=
                  \left( \begin{array}{ccccc}
                                    h & 0 & \ldots & \ldots & 0 \\
                                    h & h & 0 & \ldots & 0 \\
                                    h & h & h & 0 & \vdots \\
                                    \vdots & \ddots & \ddots & \ddots & \vdots \\
                                    h & h & \ldots & h & h
                                    \end{array} \right) \ts.
$$
Note that $\E_{n}$ represents the linear part of the explicit Euler method (without the initial step) and its inverse can be interpreted as the simplest numerical integration.
Second, if $\A$ is a regular matrix and $\left\| \cdot\right\|_{\star}$ is a norm then 
$\left\| \u\right\|_{\A,\star}=\left\| \A\u\right\|_{\star}$ defines a norm.
Then clearly
$$\left\|\u_{n} \right\|_{\$}=\left\|\E_{n}^{-1}\u_{n} \right\|_{\infty}=\left\|\u_{n} \right\|_{\E_{n}^{-1},\infty}\ts.
$$
\smallskip

It is known that weakly and strongly stable linear multistep methods are stable in the norm pair $\left( \left\|\cdot\right\|_{k\infty}\ts, \left\|\cdot\right\|_{k\infty}\right)$, cf. \cite{M1}. Moreover, strongly stable methods are stable in the $\left( \left\|\cdot\right\|_{k\infty}\ts, \left\|\cdot\right\|_{k\$}\right)$ norm pair, see \cite{M2}.
These are positive results and there is a natural question: are weakly stable methods stable in the $\left( \left\|\cdot\right\|_{k\infty}\ts, \left\|\cdot\right\|_{k\$}\right)$ norm pair or not? The following section is devoted to answer this question.

\section{Spijker's example and its extension}

First we recall Spijker's example, cf. \cite[Example 2 in Section 2.2.4]{S73}.
\begin{thm}\label{thm:S}
The explicit two-step midpoint method \eqref{midpoint} is not stable in the $\left( \left\|\cdot\right\|_{2\infty}\ts, \left\|\cdot\right\|_{2\$}\right)$ norm pair.
\end{thm}
For the sake of completeness we append the proof.

\begin{proof}
We focus on the explicit two-step midpoint method \eqref{midpoint} and rewrite it to fit into our framework. $k=2$ and 
$F_N(\u_N) =$
\begin{equation*}
                \left( \begin{array}{ccccc}
                1 & 0 & 0 & \ldots  & 0 \\
                0 & 1 & 0 & \ldots  & 0 \\
                -\frac{1}{2h} & 0 & \frac{1}{2h} &  \ldots & 0 \\
                \vdots & \ddots & \ddots &  \ddots & \vdots \\
                0 & \ldots &  -\frac{1}{2h} & 0 & \frac{1}{2h}
                \end{array} \right) 
                \left( \begin{array}{c}
                  u_0  \\
                  u_1 \\
                  u_2 \\
                  \vdots \\
                  u_{N}
                  \end{array} \right) -
				\left( \begin{array}{ccccc}
                0 & 0 & 0 & \ldots  & 0 \\
                0 & 0 & 0 & \ldots  & 0 \\
                0 & 1 & 0 &  \ldots & 0 \\
                \vdots & \ddots & \ddots &  \ddots & \vdots \\
                0 & \ldots &  0 & 1 & 0
                \end{array} \right)
                \left( \begin{array}{c}
                f_0  \\
                f_1  \\
                f_{2} \\
                \vdots \\
                f_{N}
                \end{array} \right)-
                \left( \begin{array}{c}
                                c_0  \\
                                c_1  \\
                                0 \\
                                \vdots \\
                                0
                                \end{array} \right)
\end{equation*}
using its matrix-vector form.

The goal is to show that this method is not stable in the $\left( \left\|\cdot\right\|_{2\infty}\ts, \left\|\cdot\right\|_{2\$}\right)$ norm pair i.e.
\begin{equation*}
\left\| \u_N- \v_N \right\|_{2\infty} \leq S \left\| F_N (\u_N)- F_N (\v_N) \right\|_{2\$}
\end{equation*}
does not hold.
Slightly modifying the original construction we define $f\equiv 0$, $\v_N= \0$ and
\begin{equation*}
u_{l}=
\begin{cases}
&0\quad \mbox{, if}\quad l=0,1\\
&(l-1) (-1)^{l}\quad \mbox{, if}\quad l=2,\ldots, n+1
\end{cases}
\end{equation*}
thus, $\u_N= (0,0,1,-2,3,-4,\dots)^{T}$. With this choice
$$ \left\| \u_N- \v_N \right\|_{2\infty}=\left\| \u_{n}\right\|_{\infty}=n \quad \mbox{and} \quad \left\| F_N (\u_N)- F_N (\v_N) \right\|_{2\$}= \left\| \A_{n} \u_{n}\right\|_{\$},
$$
where we can calculate
$$ \A_{n} \u_{n}=\frac{1}{h}\ts\cdot\ts\left(\frac{1}{2}, -1, 1, -1, 1, \ldots\right)^{T}\quad \mbox{thus}\quad \left\| \A_{n} \u_{n}\right\|_{\$}=\frac{1}{2}\ts.
$$
This means that the stability estimate does not hold.
\end{proof}

In the following we present the extension of this result.
\begin{thm}\label{extS}
Weakly stable methods are not stable in the $\left( \left\|\cdot\right\|_{k\infty}\ts, \left\|\cdot\right\|_{k\$}\right)$ norm pair.
\end{thm}
\begin{proof}
We assume that the method is weakly stable, thus we assume that $|\xi_{2}|=1$, $\xi_{2}\neq 1$. We set $f\equiv 0$, $\v_N=\0$ and $\u_{N}=(0,\ldots,0,u_k,\ldots,u_{n+k-1})^{T}$. For this setting stability \eqref{Sstability} is simplified to
\begin{equation*}
\left\| \u_{n}\right\|_{\infty} \leq S \left\| \A_{n}\u_{n}\right\|_{\$}\ts.
\end{equation*}
For all $S$ and for all $n_0$ we will present a vector $\u_{n}$, $n>n_0$ for which
\begin{equation}\label{ford}
\left\| \u_{n}\right\|_{\infty} > S \left\| \A_{n}\u_{n}\right\|_{\$}\ts.
\end{equation}
Note that
\begin{equation*} %\label{prod}
h\A_{n}=\alpha_0\prod\limits_{i=1}^{k}\left( \I-\xi_i\H_{n}\right)\ts,
\end{equation*}
where $\I\in\R^{n\times n}$ stands for the identity matrix, $\H_{n}\in \R^{n\times n}$ is defined as
$$
\H_{n}=
\left( \begin{array}{ccccc}
                  0 & 0 & \ldots & \ldots & 0 \\
                  1 & 0 & 0 & \ldots & 0 \\
                  0 & 1 & 0 & 0 & \ldots \\
                  \vdots & \ddots & \ddots & \ddots & \vdots \\
                  0 & \ldots & 0 & 1 & 0
                  \end{array} \right) \ts,
$$
and $\xi_i$, $i=1,\ldots,k$ are the roots of the first characteristic polynomial. This comes from the following calculation.
\begin{align*}
&h\A_{n}=\alpha_0\I+\alpha_1\H_{n}+\alpha_2\H_{n}^{2}+\ldots+\alpha_k\H_{n}^{k}= \alpha_k\prod\limits_{i=1}^{k}(\H_{n}-\nu_i\I)=\\
&\alpha_k(-1)^{k}\left( \prod\limits_{i=1}^{k}\nu_i\right) \prod\limits_{i=1}^{k}\left( \I-\dfrac{1}{\nu_i}\H_{n}\right) = \alpha_0\prod\limits_{i=1}^{k}\left( \I-\dfrac{1}{\nu_i}\H_{n}\right)= \alpha_0\prod\limits_{i=1}^{k}\left( \I-\xi_i\H_{n}\right)\ts,
\end{align*}
where we exploited the commutativity of the terms $\left( \I-\xi_i\H_{n}\right)$ and that $\xi_i=\frac{1}{\nu_i}$ since $\alpha_0+\alpha_1 z+\alpha_2 z^{2}+\ldots+\alpha_k z^{k}$ is the reciprocal polynomial of $\varrho$. This covers the case when $\forall \xi_i\neq 0$. If $\exists \xi_i= 0$ the modification of the calculation is straightforward.

Let us introduce $\w_{n}=(w_1,\ldots,w_{n})^{T}\in\R^{n}$, $\w_{n}=\E_{n} \u_{n}$. With this \eqref{ford} is equivalent to
\begin{equation*}
\left\| \w_{n}\right\|_{\$}=\left\| \E_{n}^{-1}\w_{n}\right\|_{\infty}=\left\| \u_{n}\right\|_{\infty} > S \left\| \alpha_0\prod\limits_{i=2}^{k}\left( \I-\xi_i\H_{n}\right)\w_{n}\right\|_{\$}\ts.
\end{equation*}

If $\xi_{2}=-1$ then
$$\left\| \prod\limits_{i=2}^{k}\left( \I-\xi_i\H_{n}\right)\w_{n}\right\|_{\$}\leq 
\left( \prod\limits_{i=3}^{k} \left\|\left( \I-\xi_i\H_{n}\right)\right\|_{\$}\right) 
\left\| \left( \I-\xi_{2} \H_{n}\right)\w_{n}\right\|_{\$}\leq 
2^{k-1} \left\| \left( \I-\xi_{2} \H_{n}\right)\w_{n}\right\|_{\$}\ts,
$$
since 
$$ \left\|\left( \I-\xi_i\H_{n}\right)\right\|_{\$}=
\max_{\left\|\u\right\|_{\$}=1} \left\|\left( \I-\xi_i\H_{n}\right)\u\right\|_{\$} \leq 
1+ \max_{\left\|\u\right\|_{\$}=1}  \left\|\H_{n}\u\right\|_{\$}\leq 2\ts.$$
Now, let us choose $w_m=m \ts \xi_{2}^{m}=m (-1)^{m}$.
$$\left( \left( \I-\xi_{2}\H_{n}\right)\w_{n} \right)_{m}=\xi_{2}^{m}=(-1)^{m}\ts,
$$
thus its norm
$$h\max_{1\leq l\leq n} 
\left| \sum\limits_{i=1}^{l} (-1)^{m} \right| \to 0\ts,
$$
as $h\to 0$, while
$$\left\| \w_{n}\right\|_{\$}=
h\max_{1\leq l\leq n} \left| \dfrac{l \xi_{2}^{l+1}}{\xi_{2}-1} - \dfrac{\xi_{2}^{l+1}-\xi_{2}}{ (\xi_{2}-1)^{2} } \right|=
h\max_{1\leq l\leq n} \left| \dfrac{l (-1)^{l}}{2} - \dfrac{(-1)^{l+1} + 1}{ 4 } \right|\geq 
\dfrac{h(n-1)}{2} \to \dfrac{1}{2}\ts.
$$
\smallskip

Else $\xi_{2}=e^{i\varphi}$ with $0<\varphi<\pi$ and then $\xi_{3}=e^{-i\varphi}$. The right side can be estimated similarly as before:
\begin{equation*}
\begin{split}
&\left\| \prod\limits_{i=2}^{k}\left( \I-\xi_i\H_{n}\right)\w_{n}\right\|_{\$}\leq 
\left( \prod\limits_{i=4}^{k} \left\|\left( \I-\xi_i\H_{n}\right)\right\|_{\$}\right) 
\left\| \left( \I-\xi_{2} \H_{n}\right)\left( \I-\xi_{3} \H_{n}\right)\w_{n}\right\|_{\$}\leq \\
&2^{k-2} \left\| \left( \I-2\cos \varphi \H_{n} + \H_{n}^{2} \right)\w_{n}\right\|_{\$}\ts.
\end{split}
\end{equation*}
Now, let us choose $w_m=m \ts\Re \xi_{2}^{m}$, where $\Re$ is the notation for the real part.
$$\left( \left( \I-2\cos \varphi \H_{n} + \H_{n}^{2} \right)\w_{n} \right)_{m}=
\begin{cases}
\cos \varphi\ts, &\mbox{if } m=1\\
\cos m\varphi -\cos (m-2)\varphi\ts, &\mbox{if } m\geq 2
\end{cases}
$$
thus its norm
$$h\max_{1\leq l\leq n} 
\left\lbrace \left|\cos \varphi\right|, \left|\cos l\varphi +\cos (l-1)\varphi -1\right| \right\rbrace \to 0\ts,
$$
as $h\to 0$, while $\xi_{2}^{l}$ is either periodic with period $\geq 3$ or dense on the unit circle which means that $\exists c>0$ such that
$$\left\| \w_{n}\right\|_{\$}=
h\max_{1\leq l\leq n} \left|\Re\left(  \dfrac{l \xi_{2}^{l+1}}{\xi_{2}-1} - \dfrac{\xi_{2}^{l+1}-\xi_{2}}{ (\xi_{2}-1)^{2} }\right)  \right|\geq 
h\max_{1\leq l\leq n} l\left|\Re\left(  \dfrac{\xi_{2}^{l+1}}{\xi_{2}-1} \right)  \right| - \dfrac{2h}{ |\xi_{2}-1|^{2} }>c\ts,
$$
if $n$ is large enough. This proves the statement.
\end{proof}

\section{Concluding discussion}

We conclude the paper adding a critical remark. Although Theorem \ref{extS} clearly showed the difference between weakly and strongly stable LMMs the practical side of this result is not clear at all. Stability is only a partial achievement, no doubt an important one, however, we are mostly interested in the convergence of methods.

Simply speaking the problem is the following. A weakly stable method is stable in the norm pair $\left( \left\|\cdot\right\|_{k\infty}\ts, \left\|\cdot\right\|_{k\infty}\right)$ resulting convergence in the norm $\left\|\cdot\right\|_{k\infty}$. For a strongly stable method we can obtain convergence in the same norm not depending on which type of stability ($\left( \left\|\cdot\right\|_{k\infty}\ts, \left\|\cdot\right\|_{k\infty}\right)$ or $\left( \left\|\cdot\right\|_{k\infty}\ts, \left\|\cdot\right\|_{k\$}\right)$) we use.

The profit is shifted to the consistency check.
Note that consistency in the norm $\left\|\cdot\right\|_{k\infty}$ with order $m$ implies consistency in the norm $\left\|\cdot\right\|_{k\$}$ with the same order $m$ or higher. This means that for strongly stable methods we have the freedom to check consistency in the $\left\|\cdot\right\|_{k\$}$ norm. It is a technical gain, see the tricky example \cite[Example 1 in Section 2.2.4]{S73}:
\begin{equation*}
\left(F_n(\u_n)\right)_i =
\begin{cases}
& u_0-c_0 \quad \mbox{, if} \quad i=0 \ts, \\[8pt]
&\dfrac{u_i-u_{i-1}}{h}-f_{i-1} \quad \mbox{, if} \quad 1\leq i\leq n \quad \mbox{odd}\ts,\\[8pt]
&\dfrac{u_i-u_{i-1}}{h}-f_{i} \quad \mbox{, if} \quad 2\leq i\leq n \quad \mbox{even}\ts.
\end{cases}
\end{equation*}
This one-step method is consistent of order 2 with respect to the $\left\|\cdot\right\|_{1\$}$ norm. To get consistency of order 2 with respect to the $\left\|\cdot\right\|_{1\infty}$ norm is less straightforward (however, it is possible).

Consequently, this freedom could be a technical gain. Unfortunately, not more, we can not win an order this way.

\end{document}